\documentclass[amssymb, 11pt, leqno]{amsart}
\usepackage{amsmath}
\usepackage{amsthm}
\usepackage{amsfonts}

\usepackage{latexsym}

\newcommand{\qb}[2]{{\left [{#1 \atop #2} \right]}}
\newlength{\standardunitlength}
\newtheorem{prop}{Proposition}[section]
\newtheorem{defn}[prop]{Definition}

\newtheorem{cor}[prop]{Corollary}
\newtheorem{theorem}[prop]{Theorem}
\newtheorem{example}[prop]{Example}

\title {Marking and shifting a part in partition theorems }

\author{Kathleen O'Hara}
\address{4356 Preston Forest Drive\\
Blacksburg,  VA 24060}\email{ohara.kathy1@gmail.com}

\author{Dennis Stanton}
\address{School of Mathematics\\
University of Minnesota\\
Minneapolis, MN 55455} \email{stanton@math.umn.edu}

\keywords{}

\date{September 7, 2018}

\begin{document}

\begin{abstract} Refined versions, analytic and combinatorial, are given for 
classical integer partition theorems. The examples include the Rogers-Ramanujan 
identities, the G\"ollnitz-Gordon identities, Euler's odd=distinct theorem, and 
the Andrews-Gordon identities. Generalizations of each of these theorems are given where a 
single part is ``marked" or weighted. This allows a single 
part to be replaced by a new larger part, ``shifting" a part, 
and analogous combinatorial results are given in each case. 
Versions are also given for marking a sum of parts.
\end{abstract}

\maketitle

\section{Introduction}

Many integer partition theorems can be restated as an analytic identity, as a sum 
equal to a product. One such example is the first Rogers-Ramanujan identity
\begin{equation}
\label{RR1}
\frac{1}{\prod_{k=0}^\infty (1-q^{5k+1})(1-q^{5k+4})}=
1+\sum_{k=1}^{\infty} q^{k^2}\frac{1}{(1-q)(1-q^2)\cdots (1-q^k)}.
\end{equation}

MacMahon's combinatorial version of  \eqref{RR1} 
uses integer partitions. The left side is the generating function 
for all partitions whose parts are congruent to $1$ or $4\mod 5.$
The factor $1/(1-q^9)$ on the left side allows an arbitrary number of $9$'s in an integer 
partition.  
If we ``mark" or weight the $9$ by an $x$, the factor $1/(1-q^9)$ is replaced by
$$
\frac{1}{1-xq^9}.
$$

One may ask how the right side is modified upon marking a part, and whether a 
refined combinatorial interpretation exists.

The result is known \cite[(2.2)]{O:S}, and there is a refined combinatorial version. 
The key to the combinatorial result is that the terms in the sum side are 
positive as power series in $q$ and $x$.

\begin{theorem} 
\label{RR1x}
Let $M\ge 1$ be any integer congruent to $1$ or $4\mod 5$. Then
$$
\begin{aligned}
\frac{1-q^M}{1-xq^M}&\frac{1}{\prod_{k=0}^\infty (1-q^{5k+1})(1-q^{5k+4})}=
1+q\frac{1+q+\cdots+q^{M-2}+xq^{M-1}}{1-xq^M}\\
+&\sum_{k=2}^{\infty} q^{k^2} \frac{1+q+\cdots +q^{M-1}}{1-xq^M}\frac{1}{(1-q^2)(1-q^3)\cdots (1-q^k)}.
\end{aligned}
$$
\end{theorem}
\vskip10pt\noindent

Here is a combinatorial version of Theorem~\ref{RR1x}.

\begin{theorem} 
\label{RRxcomb1}
Let $M$ be positive integer which is congruent to $1$  or $4\mod 5$. Then the
number of partitions of $n$ into parts congruent to $1$ or $4\mod 5$ with exactly $k$ $M$'s
is equal to the number of partitions $\lambda$ of $n$ with difference at 
least $2$ and
\begin{enumerate}
\item if $\lambda$ has one part, then $\lfloor n/M\rfloor=k,$
\item if $\lambda$ has at least two parts, then $\lfloor(\lambda_1-\lambda_2-2)/M\rfloor=k.$
\end{enumerate}
\end{theorem}

The purpose of this paper is to give the analogous results for several other classical 
partition theorems: the G\"ollnitz-Gordon identities, Euler's odd=distinct theorem, and 
the Andrews-Gordon identities. The main engine, Proposition~\ref{mainprop}, may be 
applied to many other single sum identities. The results obtained here by marking a part are 
refinements of the corresponding classical results.

We shall also consider ``shifting" a part, for example replacing 
all $9$'s by $22$'s in \eqref{RR1}. This is replacing the factor
$$
\frac{1}{1-q^9} \qquad {\text{ by }} \frac{1}{1-q^{22}}.
$$
We shall see that the set of partitions enumerated by the sum side is an explicit subset
of the partitions in the original identity.

Finally in section 6 we consider marking a sum of parts. 
We can extend Theorem~\ref{RRxcomb1} to allow other values of $M$, for example 
$M=7$, by marking the partition $6+1$. See Corollary~\ref{extendRRcomb}.

We use the standard notation,
$$
(A;q)_k=\prod_{j=0}^{k-1}(1-Aq^j), \qquad [M]_q=\frac{1-q^M}{1-q}.
$$
If the base $q$ is understood we may write $(A;q)_k$ as $(A)_k.$

\section{The Rogers-Ramanujan identities}

In this section we give prototypical examples for the Rogers-Ramanujan identities.

First we state  a marked version of the second Rogers-Ramanujan identity, which follows
from Proposition~\ref{mainprop}.

\begin{theorem} 
\label{RR2x}
Let $M\ge 2$ be any integer congruent to $2$ or $3\mod 5$. Then
$$
\begin{aligned}
&\frac{1-q^M}{1-xq^M}\frac{1}{(q^2;q^5)_\infty (q^3;q^5)_\infty}=
1+q^2\frac{[M-2]_q+xq^{M-2}+q^{M-1}}{1-xq^M}+\\
&\sum_{k=2}^{\infty} q^{k^2+k} \frac{[M]_q}{1-xq^M}\frac{1}{(q^2;q)_{k-1}}.
\end{aligned}
$$
\end{theorem}

Here is a combinatorial version of Theorem~\ref{RR2x}.

\begin{theorem} 
\label{RRx2comb1}
Let $M$ be positive integer which is congruent to $2$  or $3\mod 5$. Then the
number of partitions of $n$ into parts congruent to $2$ or $3\mod 5$ with exactly $k$ $M$'s
is equal to the number of partitions $\lambda$ of $n$ with difference at 
least $2$, no $1$'s, and
\begin{enumerate}
\item if $\lambda$ has one part, then $n=Mk+j$, $2\le j\le M-1$, or $j=0$ or $j=M+1,$
\item if $\lambda$ has at least two parts, then $\lfloor(\lambda_1-\lambda_2-2)/M\rfloor=k.$
\end{enumerate}
\end{theorem}

\begin{proof} We simultaneously prove Theorems~\ref{RRxcomb1} and \ref{RRx2comb1}.
We need to understand the combinatorics of the replacement in the $k^{th}$ term 
on the sum side
\begin{equation}
\label{ones}
\frac{1}{1-q}\rightarrow \frac{[M]_q}{1-xq^M}=\sum_{p=0}^\infty q^p x^{\lfloor p/M\rfloor}.
\end{equation}
In the classical Rogers-Ramanujan identities, the factor $1/(1-q)$ represents 
the difference in the first two parts after the double staircase has been removed. 
This is the second case of each theorem.
\end{proof}

\begin{example} Let $k=2$, $M=7$, and $n=22.$ The 
equinumerous sets of partitions for Theorem~\ref{RRx2comb1} are
$$
\{(8,7,7), (7,7,3,3,2), (7,7,2,2,2,2) \}
\leftrightarrow \{ (22), (20,2), (19,3)\}.
$$
\end{example}

Equivalent combinatorial versions of Theorems~\ref{RRxcomb1} and \ref{RRx2comb1} 
may be given (see \cite[Theorem 2, Theorem 3]{O:S}). This time the terms $k\ge M$ of 
the sum side are considered, and the 
replacement considered is
$$
\frac{1}{1-q^M}\rightarrow \frac{1}{1-xq^M},
$$
namely the part $M$ is marked on the sum side.
We need notation for when a double staircase is 
removed from a partition with difference at least two.

\begin{defn}
For any partition 
$\lambda$ with $k$ parts whose difference of parts is at least 2, let $\lambda^*$ 
denote the partition obtained upon removing the double staircase $(2k-1,2k-3,\cdots ,1)$ from 
$\lambda$, and reading the result by columns.

For any partition
$\lambda$ with $k$ parts and no 1's whose difference of parts is at least 2, 
let $\lambda^{**}$ denote the partition obtained upon removing the double 
staircase $(2k,2k-2,\cdots ,2)$ from $\lambda$, and reading the result by columns.  
\end{defn}

\begin{theorem} 
\label{RRxcomb2}
Let $M$ be positive integer which is congruent to $1$  or $4\mod 5$. Then the
number of partitions of $n$ into parts congruent to $1$ or $4\mod 5$ with exactly $k$ $M$'s
is equal to the number of partitions $\lambda$ of $n$ with difference at 
least $2$ and
\begin{enumerate}
\item if $\lambda$ has one part, then $\lfloor n/M\rfloor=k,$
\item if $\lambda$ has between two and $M-1$ parts, then $\lfloor(\lambda_1-\lambda_2-2)/M\rfloor=k$,
\item if $\lambda$ has at least $M$ parts, then $\lambda^*$ has exactly $k$ $M$'s.
\end{enumerate}
\end{theorem}

\begin{example} Let $k=2$, $M=4$, and $n=24.$ The 
equinumerous sets of partitions for Theorem~\ref{RRxcomb2} are
$$
\begin{aligned}
&\{(16,4^2), (14,4^2,1^2), (11,4^2,1^5), (9,6,4^2,1), (9,4^2,1^7), (6,6,4^2,1^4),\\
&(6,4^2,1^{10}),(4^2,1^{16})\}\leftrightarrow\\
& \{(9,7,5,3), (18,5,1), (17,6,1), (17,5,2), (16,6,2), (16,5,3), (17,7), (18,6)\}.
\end{aligned}
$$
\end{example}

\begin{theorem} 
\label{RRx2comb2}
Let $M$ be positive integer which is congruent to $2$  or $3\mod 5$. Then the
number of partitions of $n$ into parts congruent to $2$ or $3\mod 5$ with exactly $k$ $M$'s
is equal to the number of partitions $\lambda$ of $n$ with difference at 
least $2$, no $1$'s and
\begin{enumerate}
\item if $\lambda$ has one part, then $n=Mk+j$, $2\le j\le M-1$, or $j=0$ or $j=M+1,$
\item if $\lambda$ has between two and $M-1$ parts, then $\lfloor(\lambda_1-\lambda_2-2)/M\rfloor=k$,
\item if $\lambda$ has at least $M$ parts, then $\lambda^{**}$ has exactly $k$ $M$'s.
\end{enumerate}
\end{theorem}

\section{A general expansion}

In this section we give a general expansion, Proposition~\ref{mainprop}, 
for marking a single part. 

Many partition identities have a sum side of the form
$$
\sum_{j=0}^\infty \frac{\alpha_j}{(q;q)_j},
$$
where $\alpha_j$ has non-negative coefficients as a power series in $q$.

These include \begin{enumerate}
\item the Rogers-Ramanujan identities, $\alpha_j=q^{j^2} {\text{ or }} q^{j^2+j},$
\item Euler's odd=distinct theorem, $\alpha_j=q^{\binom{j+1}{2}},$
\item the G\"ollnitz-Gordon identities, $q$ replaced by $q^2$, $\alpha_j=q^{j^2}(-q;q^2)_j,$
\item all partitions by largest part, $\alpha_j=q^j,$
\item all partitions by Durfee square, $\alpha_j=q^{j^2}/(q;q)_j.$
\end{enumerate}

A part of size $M$ may be marked in general using the next proposition. 

\begin{prop} 
\label{mainprop}
For any positive integer $M$, if $\alpha_0=1$,
$$
\frac{1-q^M}{1-wq^M}\sum_{j=0}^\infty \frac{\alpha_j}{(q;q)_j}=
1+\frac{\alpha_1[M]_q-q^M+wq^M}{1-wq^M}+\sum_{j=2}^\infty 
\frac{[M]_q}{1-wq^M}\frac{\alpha_j}{(q^2;q)_{j-1}}.
$$
\end{prop}

As long as $\alpha_1$ has the property that 
$$
\alpha_1 [M]_q-q^M
$$
is a positive power series in $q$, the right side has a combinatorial interpretation. 

There are two possible elementary 
combinatorial interpretations. For any $j\ge 2,$ the factor
$$
\frac{[M]_q}{1-wq^M}=\sum_{p=0}^\infty q^p w^{[p/M]}
$$
replaces  $1/(1-q),$ which accounts for parts of size $1$ in a partition. 
This is a weighted form of the number of 1's.

The second interpretation holds for terms with $j\ge M.$ Here
$$
\frac{[M]_q}{1-wq^M}\frac{1}{(q^2;q)_{j-1}}=
\frac{1}{(1-q)\cdots (1-q^{M-1})(1-wq^M)(1-q^{M+1})\cdots (1-q^j)}.
$$
In this case the part of size $M$ is marked by $w$.

For a particular combinatorial application of Proposition~\ref{mainprop} one must 
realize what the denominator factors
$(1-q)$ and $(1-q^M)$ represent on the sum side. For example, in the Rogers-Ramanujan identities 
these factors account for $1$'s and $M$'s in $\lambda^*.$ Since 
$$
(\#1's {\text{ in }}\lambda^*)=\lambda_1-\lambda_2-2,
$$
the two 
interpretations are Theorem~\ref{RRxcomb1} and Theorem~\ref{RRxcomb2}.

\subsection{Distinct parts} 

Choosing $\alpha_j=q^{\binom{j+1}{2}}$ in Proposition~\ref{mainprop} gives 
distinct partitions, which by 
Euler's theorem are equinumerous with partitions into odd parts. Here is the marked version.

\begin{cor} 
\label{RR1cor}
For any odd positive integer $M$,
$$
\begin{aligned}
&\frac{1}{(1-q)(1-q^3)\cdots(1-q^{M-2})(1-wq^M)(1-q^{M+2})\cdots}\\
&=
1+\frac{q+q^2+\cdots+q^{M-1}+wq^M}{1-wq^M}+\sum_{j=2}^\infty 
\frac{q^{\binom{j+1}{2}}}{(q^2;q)_{j-1}}\frac{[M]_q}{1-wq^M}.
\end{aligned}
$$
\end{cor}

\begin{defn}
For any partition $\lambda$ with $j$ distinct parts let $\lambda^{St}$ be the 
partition obtained upon 
removing a staircase $(j,j-1,\cdots ,1)$ from $\lambda$, and reading the result by columns.
\end{defn}

\begin{example} If $\lambda=(8,7,3,1)$, then $\lambda^{St}=(3,2,2,2).$
\end{example}

Here is the combinatorial version of Corollary~\ref{RR1cor}, generalizing Euler's theorem. 

\begin{theorem} 
\label{combthm2}
For any odd positive integer $M$, the number 
of partitions of $n$ into odd parts with exactly $k$ parts of size $M$, is equal to 
the number of partitions $\lambda$ of $n$ into distinct parts such that 
\begin{enumerate}
\item if $\lambda$ has one part, then $\lfloor n/M\rfloor=k,$
\item if $\lambda$ has at least two parts, then $\lfloor (\lambda_1-\lambda_2-1)/M\rfloor=k.$
\end{enumerate}
\end{theorem}

\begin{example} Let $k=2$, $M=5$, and $n=18.$ The 
equinumerous sets of partitions for Theorem~\ref{combthm2} are
$$
\begin{aligned}
&\{(7,5,5,1),(5,5,3,3,1,1),(5,5,3,1^5), (5,5,1^8)\}\\
&\leftrightarrow
\{(16,2),(15,3), (15,2,1), (14,3,1)\}.
\end{aligned}
$$

\end{example}

\begin{prop} There is an $M$-version of the Sylvester ``fishhook"  bijection which proves 
Theorem~\ref{combthm2}. 
\end{prop}

\begin{proof} Let $FH$ be the fishhook bijection from partitions with distinct parts to 
partitions with odd parts. If $FH(\lambda)=\mu,$ it is known that the number of 
$1$'s in $\mu$ is $\lambda_1-\lambda_2-1$, except for $FH(n)= 1^n.$ This proves  
Theorem~\ref{combthm2} if $M=1$, and $FH$ is the bijection for $M=1$.

For the $M$-version, $M>1$, let $\lambda$ have distinct parts. For $\lambda=n$ a single part, 
Define the $M$-version by $FH^M(n)= (M^k, 1^{n-kM})$ which has $k$ parts of size $M$.
Otherwise $\lambda$ has at least two parts, and 
$$
kM\le \lambda_1-\lambda_2-1\le (k+1)M-1.
$$
Let $\theta$ be the partition with distinct parts where $\lambda_1$ has 
been reduced by $kM,$
$$
0\le \theta_1-\theta_2-1\le M-1.
$$
Finally put $\gamma=FH(\theta),$ and note that $\gamma$ has at most $M-1$ $1$'s.

There are 2 cases. If $\gamma$ has no parts of size $M$, define 
$FH^M(\lambda)=\gamma\cup M^k$, so that $FH^M(\lambda)$ is a 
partition with odd parts, exactly $k$ parts of size $M$, and at most $M-1$ $1$'s.

If $\gamma$ has $r\ge 1$ parts of size $M,$ change all of them to $rM$ $1$'s to obtain
$\gamma'$ with at least $M$ $1$'s.  Then put $FH^M(\lambda)=\gamma'\cup M^k,$ 
so that $FH^M(\lambda)$ is a partition with odd parts, exactly $k$ parts of size $M$, 
and at least $M$ $1$'s.
\end{proof}

\begin{theorem} 
\label{combthm3}
For any odd positive integer $M$, the number 
of partitions of $n$ into odd parts with exactly $k$ parts of size $M$, is equal to 
the number of partitions $\lambda$ of $n$ into distinct parts such that 
\begin{enumerate}
\item if $\lambda$ has one part, then $\lfloor n/M \rfloor=k,$
\item if $\lambda$ has between two and $M-1$ parts, then $\lfloor(\lambda_1-\lambda_2-1)/M\rfloor=k,$ 
\item if $\lambda$ has at least $M$ parts, then $\lambda^{St}$ has exactly $k$ $M$'s.
\end{enumerate}
\end{theorem}

\begin{example} Let $k=2$, $M=3$, and $n=18.$ The 
equinumerous sets of partitions for Theorem~\ref{combthm3} are
$$
\begin{aligned}
&\{(11,3,3,1),(9,3,3,1^3), (7,5,3,3), (7,3,3,1^5),(5,5,3,3,1,1),(5,3,3,1^7), (3,3,1^{12})\}\\
&\leftrightarrow
\{(7,6,4,1),(8,5,4,1),(11,4,3),(10,5,3),(9,6,3),(8,7,3),(13,5)\}.
\end{aligned}
$$
\end{example}

%
\subsection{G\"ollnitz-Gordon identities}

The G\"ollnitz-Gordon identities are (see \cite{And1}, \cite{Gol}, \cite{Gor})
\begin{equation}
\label{GGeq1}
\sum_{n=0}^\infty q^{n^2}\frac{(-q;q^2)_n}{(q^2;q^2)_n}=
\frac{1}{(q;q^8)_\infty (q^4;q^8)_\infty (q^7;q^8)_\infty},
\end{equation}

\begin{equation}
\sum_{n=0}^\infty q^{n^2+2n}\frac{(-q;q^2)_n}{(q^2;q^2)_n}=
\frac{1}{(q^3;q^8)_\infty (q^4;q^8)_\infty (q^5;q^8)_\infty}
\end{equation}

We apply Proposition~\ref{mainprop} with $q$ replaced by $q^2,$ $M$ replaced 
by $M/2,$ and 
$\alpha_j=q^{j^2}(-q;q^2)_j$ to obtain the next result. 

\begin{cor} 
\label{GGx}
Let $M$ be a positive integer. Then
$$
\begin{aligned}
\frac{1-q^M}{1-wq^M}&\frac{1}{(q;q^8)_\infty (q^4;q^8)_\infty (q^7;q^8)_\infty}\\
&=1+\frac{q[M-1]_q+wq^M}{1-wq^M}
+\sum_{j=2}^\infty q^{j^2}\frac{[M]_{q}}{1-wq^M}\frac{(-q^3;q^2)_{j-1}}{(q^4;q^2)_{j-1}}.
\end{aligned}
$$
\end{cor}

We used 
$$
\frac{q(1+q)[M/2]_{q^2}-q^M+wq^M}{1-wq^M}=\frac{q[M-1]_q+wq^M}{1-wq^M}
$$
to simplify the second term in the sum in Corollary~\ref{GGx}. 
Note that the numerator has positive coefficients, and thus a simple combinatorial interpretation.

Here is the combinatorial restatement \cite[Theorem 2]{Gor} 
of the first G\"ollnitz-Gordon identity.

\begin{theorem} 
\label{GGcomb}
The number of partitions of $n$ 
into parts congruent to $1,4,{\text{ or }}7\mod 8$
is equal to the number of partitions of $n$ into parts whose 
difference is at least 2, and greater than 2 for consecutive even parts.
\end{theorem}

For the combinatorial version of Corollary~\ref{GGx}, we need to recall why 
the sum side of \eqref{GGeq1} is the generating function for the restricted 
partitions with difference at least 2.  In particular we must identify 
what the denominator factor $1-q$ represents in the sum side.

Suppose $\lambda$ is such a partition with $j$ parts. 
This is equivalent to showing that the generating function for 
$\lambda^*$ is
\begin{equation}
\label{GGfactor}
\frac{(-q;q^2)_j}{(q^2;q^2)_j}=\frac{1+q}{1-q^2}\frac{(-q^3;q^2)_{j-1}}{(q^4;q^2)_{j-1}}.
\end{equation}
The partition $\mu=\lambda-(2j-1,2j-3,\cdots, 1)$ has at most $j$ parts, and 
the odd parts of $\mu$ are distinct. The column read version $\lambda^*=\mu^t$ can be 
built in the following way. Take arbitrary parts from sizes $j, j-1, \cdots, 1$ with even multiplicity,
whose generating function is 
$1/(q^2;q^2)_j.$ The rows now have even length. Then choose a subset 
of the odd integers $1+0, 2+1, \cdots j+(j-1).$ For each such odd part $k+(k-1)$ add 
columns of length $k$ and $k-1$. This keeps all rows even, except the $k^{th}$ 
row which is odd and distinct.

We see that the factor $(1+q)/(1-q^2)=1/(1-q)$ in \eqref{GGfactor} accounts for $1$'s in 
$\lambda^*.$ 
In Corollary~\ref{GGx} this quotient is replaced by
$$
\frac{1+q}{1-q^2}\rightarrow \frac{[M]_q}{1-wq^M}=\sum_{p=0}^\infty q^p w^{[p/M]}.
$$
There is one final opportunity for a $1$ to appear in $\lambda^*$: when $3=2+1$ is chosen as 
an odd part. This occurs only when the second part of $\lambda$ is even.

\begin{theorem} 
\label{GGxcomb}
Let $M$ be a positive integer which is congruent to $1,4$ or $7\mod 8.$
The number of partitions of $n\ge 1$ 
into parts congruent to $1,4{\text{ or }}7\mod 8$ with exactly $k$ $M$'s, 
is equal to the number of partitions $\lambda$ of $n$ into parts whose 
difference is at least 2, and greater than 2 for consecutive even parts such that
\begin{enumerate}
\item if $\lambda$ has a single part, then $[n/M]=k,$
\item if $\lambda$ has at least two parts and the second part of $\lambda$ is even, 
$$\lfloor (\lambda_1-\lambda_2-3)/M\rfloor=k,$$
\item if $\lambda$ has at least two parts and the second part of $\lambda$ is odd, 
$$\lfloor(\lambda_1-\lambda_2-2)/M\rfloor=k.$$
\end{enumerate}
\end{theorem}

\begin{example} 
Let $k=3$, $M=7$, and $n=31.$ The 
equinumerous sets of partitions for Theorem~\ref{GGxcomb} are
$$
\begin{aligned}
&\{(9,7,7,7,1), (7,7,7,4,4,1,1), (7,7,7,4,1^6), (7,7,7,1^{10})\}\leftrightarrow\\
&\{(30,1), (29,2),(28,3), (27,3,1)\}.
\end{aligned}
$$
Note that $\lambda=(27,4)$ is not allowed because the 
second part of $\lambda$ is even.
\end{example}

For the second G\"ollnitz-Gordon identity, the version of Corollary~\ref{GGx} is

\begin{equation}
\begin{aligned}
\frac{1-q^M}{1-wq^M}&\frac{1}{(q^3;q^8)_\infty (q^4;q^8)_\infty (q^5;q^8)_\infty}\\
=&1+\frac{q^3+\cdots +q^{M-1}+wq^M+q^{M+1}+q^{M+2}}{1-wq^M}\\
+&\sum_{j=2}^\infty q^{j^2+2j}\frac{[M]_{q}}{1-wq^M}\frac{(-q^3;q^2)_{j-1}}{(q^4;q^2)_{j-1}}.
\end{aligned}
\end{equation}

Here is the combinatorial refinement of \cite[Theorem 3]{Gor}.

\begin{theorem} 
\label{GGxcombM}
Let $M$ be a positive integer which is congruent to $3,4$ or $5\mod 8.$
The number of partitions of $n\ge 1$ 
into parts congruent to $3,4{\text{ or }}5\mod 8$ with exactly $k$ $M$'s, 
is equal to the number of partitions $\lambda$ of $n$ into parts whose 
difference is at least 2, greater than 2 for consecutive even parts, smallest part at least 3, such that
\begin{enumerate}
\item if $\lambda$ has a single part, then $n=Mk,$ or $n=Mk+j$, $3\le j\le M+2$, $j\neq M$,
\item if $\lambda$ has at least two parts and the second part of $\lambda$ is even, 
$$\lfloor (\lambda_1-\lambda_2-3)/M\rfloor=k,$$
\item if $\lambda$ has at least two parts and the second part of $\lambda$ is odd, 
$$\lfloor(\lambda_1-\lambda_2-2)/M\rfloor=k.$$
\end{enumerate}
\end{theorem}

\section{An Andrews-Gordon version}

The Andrews-Gordon identities are
\begin{theorem} 
\label{AG}
If $0\le a \le k$, then
$$
\frac{(q^{k+1-a},q^{k+2+a},q^{2k+3};q)_\infty}{(q;q)_\infty}=
\sum_{n_1\ge n_{2} \ge \cdots \ge n_k\ge 0} 
\frac{q^{n_1^2+n_2^2+\cdots +n_k^2+n_{k+1-a}+\cdots +n_k}}
{(q)_{n_1-n_2}\cdots (q)_{n_{k-1}-n_k}(q)_{n_k}}.
$$
\end{theorem}

The Rogers-Ramanujan identities are the cases $k=1$, $a=0,1.$

Because Theorem~\ref{AG} has a multisum instead 
of a single sum, we cannot apply Proposition~\ref{mainprop}. Nonetheless the 
same idea can be applied to obtain a marked version of Theorem~\ref{AG}. 

Let $F_k^a$ denote the right side multisum of Theorem~\ref{AG} for 
$0\le a\le k$, and let $F_k^a=F_k^0$ for $a<0.$
So we have
$$
F_k^a=F_{k-1}^{a-1}+\sum_{n_1\ge n_{2} \ge \cdots \ge n_k\ge 1} 
\frac{q^{n_1^2+n_2^2+\cdots +n_k^2+n_{k+1-a}+\cdots +n_k}}
{(q)_{n_1-n_2}\cdots (q)_{n_{k-1}-n_k}(q)_{n_k}}.
$$

Multiplying by $\frac{1-q^M}{1-xq^M}$ yields
$$
\begin{aligned}
&\frac{1-q^M}{1-xq^M} F_k^a=\frac{1-q^M}{1-xq^M} 
F_{k-1}^{a-1}\\
&+\sum_{n_1\ge n_{2} \ge \cdots \ge n_k\ge 1} 
\frac{q^{n_1^2+n_2^2+\cdots +n_k^2+n_{k+1-a}+\cdots +n_k}}
{(q)_{n_1-n_2}\cdots (q)_{n_{k-1}-n_k}(q^2;q)_{n_k-1}}\frac{[M]_q}{1-xq^M},
\end{aligned}
$$
which, upon iterating, is the following weighted version of the 
Andrews-Gordon identities.

\begin{theorem} 
\label{AGbig}
For $0\le a\le k$, let $M$ be any positive integer 
not congruent to $0$, $\pm (k+1-a)$ modulo $2k+3.$ Then
$$
\begin{aligned}
&\frac{1-q^M}{1-xq^M}\frac{(q^{k+1-a},q^{k+2+a},q^{2k+3};q)_\infty}{(q;q)_\infty}
=1+A+
\sum_{n_1=2}^\infty \frac{q^{n_1^2+B}}{(q^2;q)_{n_1-1}}\frac{[M]_q}{1-xq^M}\\
&+\sum_{r=2}^k\sum_{n_1\ge n_{2} \ge \cdots \ge n_r\ge 1} 
\frac{q^{n_1^2+n_2^2+\cdots +n_r^2+n_{k+1-a}+\cdots +n_r}}
{(q)_{n_1-n_2}\cdots (q)_{n_{r-1}-n_r}(q^2;q)_{n_r-1}}\frac{[M]_q}{1-xq^M}
\end{aligned}
$$
where 
\begin{enumerate} 
\item {\text{ for }}$0\le a<k, \quad B=0, \quad A=q([M-1]_q+xq^{M-1})/(1-xq^M)$\\
\item {\text{ for }}$a=k, \quad B=n_1, \quad A=q^2([M-2]_q+xq^{M-2}+q^{M-1})/(1-xq^M).$
\end{enumerate}
\end{theorem}

For a combinatorial version of Theorem~\ref{AGbig} we use Andrews'  
Durfee dissections, and $(k+1,k+1-a)$-admissible partitions, see \cite{And2}. 

\begin{defn} Let $k$ be a positive integer and $0\le a\le k.$
A partition $\lambda$ is called $(k+1,k+1-a)$-admissible  if
$\lambda$ may be dissected by $r\le k$ successive Durfee rectangles, 
moving down, of sizes
$$
n_1\times n_1, \cdots, n_{k-a}\times n_{k-a},\  (n_{k-a+1}+1)\times n_{k-a},
\cdots ,(n_r+1) \times n_r.
$$
such that the $(n_1+n_2+\cdots+ n_{k-a+i}+i)^{th}$ part of $\lambda$ is $
n_{k-a+i},$ for $1\le i\le r-(k-a).$
\end{defn}

Note that $r\le k-a$ is allowed, in which case all of the Durfee rectangles are squares.
Also, the parts of $\lambda$ to the right of the Durfee rectangles are not constrained, 
except at the last row of the non-square Durfee rectangle, where it is empty.

\begin{example} Suppose $k=3$ and $a=2$. 
Then $\lambda=91$ is not $(4,2)$-admissible: the Durfee square has size $n_1=1$, but the 
next Durfee rectangle of size $2\times 1$ does not exist, so the second part 
cannot be covered if $r\ge 2.$
\end{example}
 
Theorem 2 in \cite{And2} interprets Theorem~\ref{AG}.

\begin{prop} The generating function for all partitions which are $(k+1,k+1-a)$-admissible 
is given by the sum in Theorem~\ref{AG}. 
\end{prop}

We need to understand the replacement
$$
\frac{1}{(q)_{n_r}}=\frac{1}{(1-q) (q^2;q)_{n_r-1}}
\rightarrow \frac{1}{(q^2;q)_{n_r-1}}\frac{[M]_q}{1-xq^M}
$$ 
in the factor $(q)_{n_r}$ to give a combinatorial version of Theorem~\ref{AGbig}.

First we recall \cite{And2} 
that if the sizes of the Durfee rectangles are fixed by 
$n_1,n_2,\cdots ,n_r,$ then the generating function for the partitions which 
have this Durfee dissection is
$$
\frac{1}{(q)_{n_1}}\prod_{j=1}^{r-1} \qb{n_j}{n_{j+1}}_q=
\frac{1}{(q)_{n_r}}\prod_{j=1}^{r-1}\frac{1}{(q)_{n_j-n_{j+1}}}.
$$
(A simple bijection for this fact is given in \cite{Gre}.) Upon multiplying by 
$$(1-q^M)/(1-xq^M)$$ we have
$$
\frac{[M]_q}{1-xq^M}\frac{1}{(q^2;q)_{n_1-1}}\prod_{j=1}^{r-1} \qb{n_j}{n_{j+1}}_q=
\frac{[M]_q}{1-xq^M}\frac{1}{(q^2;q)_{n_r-1}}\prod_{j=1}^{r-1}\frac{1}{(q)_{n_j-n_{j+1}}}.
$$ 

Consider the factor $1/(q)_{n_1},$ which accounts for the portion of the 
partition to the right of the first Durfee rectangle of $\lambda.$ 
In this factor we are replacing
$$
\frac{1}{1-q}\rightarrow \frac{[M]_q}{1-xq^M}.
$$ 
As before, the $M$ $1$'s in the columns to the right of the first 
Durfee rectangle are weighted by $x$. These $1$'s are again a difference 
in the first two parts of $\lambda.$ 

Putting these pieces together, the following result is a 
combinatorial restatement of Theorem~\ref{AGbig}.

\begin{theorem} 
\label{AGcombM}
Fix integers $a,k,M$ satisfying $0\le a\le k$ and 
$M\not\equiv 0, \pm (k+1-a)\mod 2k+3$. 
The number of partitions of $n$ 
into parts not congruent to $0$, $\pm (k+1-a)\mod 2k+3$ with exactly $j$ $M$'s, is 
equal to the number of partitions $\lambda$ of $n$ which are $(k+1,k+1-a)$-admissible
with $r\le k$ Durfee rectangles of sizes 
$$
n_1\times n_1, \cdots, n_{k-a}\times n_{k-a}, (n_{k-a+1}+1)\times n_{k-a},
\cdots ,(n_r+1) \times n_r
$$
of the following form:
\begin{enumerate}
\item if $r=n_1=1,$ and $0\le a<k$, 
$\lambda$ is a single part of size $Mj$, $Mj+1, \cdots , Mj+(M-1)$, or
\item if $r=n_1=1,$ and $a=k$, 
$\lambda=(\lambda_1,1)$ has size $Mj$, $Mj+2, \cdots ,Mj+(M-1),$ or $Mj+(M+1)$.  
\item if  $n_1=1$ and $r\ge 2,$ then $\lfloor (\lambda_1-n_1)/M \rfloor=j,$
\item if  $n_1\ge 2,$ then $\lfloor (\lambda_1-\lambda_2)/M \rfloor=j.$
\end{enumerate}
\end{theorem}

\begin{table}[ht]
\centering
\caption{Theorem~\ref{AGcombM} when $a=2, k=3, M=3$}
\label{my-label}
\begin{tabular}{ccccl}
partition of 10 without 2,7,9 &\# of 3's& (4,2)-admissible partition of 10& value of $j$&\\
$ 10$& $0$ &$ 10 $&$3$&  \\
$811$& $0$&$61111$&$1$& \\
$64$&$0$&$421111$&$0$&  \\
$631$&$1$ &$322111$&$0 $& \\
$61111$&$0$ &$331111$&$0$&  \\
$55$&$0$&$811$ &$2$&  \\
$541$&$0$&$6211$&$1$&  \\
$5311$&$1$&$5311 $&$0$& \\
$511111$&$0$&$5221$&$1$&  \\
$4411$&$0$&$22222$&$0$&  \\
$433$&$2$&$4411$&$0$&  \\
$43111$&$1$&$4321 $&$0$& \\
$4111111$&$0$&$82$&$2$&  \\
$3331$&$3$&$73$&$1$&  \\
$331111$&$2$&$64$&$0$&  \\
$31111111$&$1$&$55$&$0$& \\
$ 1111111111$&$0$&$433$&$0$&  \\
\end{tabular}
\end{table}

\section{Shifting a part}

The weighted versions allow one to shift a part. For example in first 
Rogers-Ramanujan identity, what happens if 
parts of size $11$ are replaced by parts of size $28$? All we need to do is to choose 
$M=11$ and $x=q^{17}$ in Theorem~\ref{RR1x}.

\begin{cor} 
\label{shiftcomb}
Let $M$ be a positive integer which is congruent to $1$ or $4$ modulo $5.$ Let 
$N>M$ be an integer not congruent to $1$ or $4$ modulo $5.$  
The number of partitions of $n$ into parts congruent to $1$ or $4$ modulo $5$, except $M$, 
or parts of size $N$, is equal to the 
number of partitions $\lambda$ of $n$ with difference at least 2, such that
\begin{enumerate}
\item $\lambda$ has a single part, which is congruent to $0,1,\cdots, {\text{ or }}M-1\mod N,$
\item $\lambda$ has at least two parts, and $\lambda_1-\lambda_2-2$ is 
congruent to $0,1,\cdots, {\text{ or }}M-1\mod N.$
\end{enumerate}
\end{cor} 

\begin{example} Let $N=8$, $M=4$, and $n=9.$ The 
equinumerous sets of partitions for Corollary~\ref{shiftcomb} are
$$
\{(9),(6,1,1,1),(8,1), (1^9)\}\leftrightarrow
\{(9),(6,3), (7,2), (5,3,1)\}.
$$
\end{example}

An related example occurs when two parts are shifted:  $1$ and $4$ are replaced
by $2$ and $3$. The appropriate identity is
\begin{equation}
\begin{aligned}
&\frac{1}{(1-q^2)(1-q^3)(q^6;q)_\infty (q^9;q)_\infty}\\
=& 1+\frac{q^2(1+q)}{1-q^3}+\sum_{k=2}^\infty
\frac{q^{k^2}}{(q^2;q)_{k-1}}\frac{1+q^2}{1-q^3}.
\end{aligned}
\end{equation}

\begin{theorem} 
\label{weirdshift}
The number of partitions of $n$ into parts from 
$$
\{2,3,5k+1,5k+4: k\ge 1\}
$$ 
is equal to the number of partitions $\lambda$ of $n$ with   
difference at least $2$ and
\begin{enumerate}
\item if $\lambda$ has a single part, then $n\not\equiv 1\mod 3$,
\item if $\lambda$ has at least two parts, then ($\lambda_1-\lambda_2-2$) 
$\not\equiv 1\mod 3.$  
\end{enumerate}
\end{theorem}

\begin{example} Let $n=13$. The two equinumerous sets of partitions in 
Theorem~\ref{weirdshift} are
$$
\begin{aligned}
&\{(11,2), (9,2,2), (6,3,2,2), (3,2,2,2,2,2), (3,3,3,2,2)\}\\ 
&\leftrightarrow 
\{(12,1), (10,3), (9,4), (8,4,1), (7,5,1)\}.
\end{aligned}
$$
The possible partitions with difference at least $2$ 
$$\{(13), (11,2), (8,5), (9,3,1), (7,4,2)\}$$ are disallowed. 
\end{example}

\begin{cor} 
\label{weirdshift2}
Let $M$ be an odd positive integer. Let 
$N>M$ be an even integer.  
The number of partitions of $n$ into odd parts except $M$, 
or parts of size $N$, is equal to the 
number of partitions $\lambda$ of $n$ into distinct parts, such that
\begin{enumerate}
\item $\lambda$ has a single part, which is congruent to $0,1,\cdots, {\text{ or }}M-1\mod N,$
\item $\lambda$ has at least two parts, and $\lambda_1-\lambda_2-1$ is 
congruent to $0,1,\cdots, {\text{ or }}M-1\mod N.$
\end{enumerate}
\end{cor} 

\begin{example} If $N=8$, $M=3$, and $n=9$ the equinumerous sets 
in Corollary~\ref{weirdshift2} are
$$
\{(9), (8,1), (7,1,1), (5,1^4), (1^9)\}\leftrightarrow
\{(9), (5,4), (6,3), (5,3,1), (4,3,2)\}.
$$
\end{example}

\section{Marking a sum of parts}

One may ask if Theorems~\ref{RR1x} and ~\ref{RR2x} have combinatorial interpretations 
without the modular conditions on $M$. The sum sides retain 
the interpretations given by Theorems~\ref{RRxcomb1} and \ref{RRx2comb1} and are positive 
as a power series in $q$ and $w$. It remains to understand what the product side 
represents as a generating function of partitions. 
We give in Proposition~\ref{weirdprop} a general positive 
combinatorial expansion for the product side. 
We call this ``marking a sum of parts".

As an example suppose that $M=A+B,$ is a sum of two parts, where $A$ and $B$ 
are distinct integers congruent to $1$ or $4\mod 5.$ The quotient in 
the product side of Theorem~\ref{RR1x}
$$
\frac{1-q^{A+B}}{1-wq^{A+B}}\frac{1}{(1-q^A)(1-q^B)}=
\frac{1}{(1-q^B)(1-wq^{A+B})}+\frac{q^A}{(1-q^A)(1-wq^{A+B})}
$$   
is a generating function for partitions with parts $A$ or $B$. The first term allows
the number of $B$'s to be at least as many as the number of $A$'s. The second term
allows the number of $A$'s to be greater than the number of $B$'s. 
The exponent of $w$ is the number of times a pair $AB$ appears in a partition. 
For example, if $A=6$, $B=4$, the partition $(6,6,4,4,4,4)$ contains 
$64$ twice, along with two $4$'s. We have found a prototypical result.

\begin{prop} 
\label{proto}
Let $M=A+B$ for some $A,B\equiv 1,4 \mod 5, A\neq B.$ Then
$$
\frac{1-q^M}{1-wq^M} \frac{1}{(q;q^5)_\infty (q^4;q^5)_\infty}
$$
is the generating function for all partitions $\mu$ with parts $\equiv 1,4\mod 5$ 
by the number of occurrences of the pair $AB$
\end{prop}

A more general statement holds for partitions other than $M=A+B.$  
To state this result, we need to define an analogue of the multiplicity of a single 
part to a multiplicity of a partition.  We again use the multiplicity notation for a partition, for example
$(7^3,4^1, 2^3)$ denotes the partition $(7,7,7,4,2,2,2).$

\begin{defn} 
Let $\lambda=(A_1^{m_1}, \cdots, A_k^{m_k})$ be a partition. We say 
$\lambda$ is inside $\mu$ $k$ times, $k=E_\lambda(\mu),$ if 
$$
k=max\{j: j\ge 0, \mu {\text{ contains at least }} jm_s 
{\text{ parts of size }}A_s {\text{ for all }}s\}.
$$ 
\end{defn}

\begin{example} Let $\lambda=(6^1,4^2,1^1)$, $\mu=(9^1,6^7, 4^5, 1^8).$ Then
$E_\lambda(\mu)=2$ but not $3$ because $\mu$ contains only five $4$'s.  
\end{example}

With this definition, Proposition~\ref{proto} holds for any partition.
\begin{prop} 
\label{weirdprop}
Let $\lambda\vdash M$ be a fixed partition into parts  
congruent to $1$ or $4\mod 5.$ Then 
$$
\frac{1-q^M}{1-wq^M} \frac{1}{(q;q^5)_\infty (q^4;q^5)_\infty}
$$
is the generating function for all partitions $\mu$ into parts congruent to 
$1$ or $4\mod 5,$
$$
\sum_{\mu} q^{||\mu||} w^{E_\lambda(\mu)},
$$
where $E_\lambda(\mu)$ is the number of times $\lambda$ appears in $\mu.$ 
\end{prop}

The modular condition on the parts in Proposition~\ref{weirdprop} is irrelevant. 

\begin{prop} 
\label{weirdprop2}
Let $\mathbb{A}=\{A_1,A_2,\cdots \}$ be any set of positive integers. Suppose that 
$\lambda=(B_1^{m_1},\cdots, B_k^{m_k})$ is a partition whose parts come from $\mathbb{A}$
and $M=\sum_{i=1}^k m_i B_i.$ Then
$$
\frac{1-q^M}{1-wq^M}\prod_{i=1}^\infty (1-q^{A_i})^{-1}
$$
is the generating function for all partitions $\mu$ with parts from $\mathbb{A}$ 
$$
\sum_{\mu} q^{||\mu||} w^{E_\lambda(\mu)}.
$$
\end{prop}

\begin{proof}
We start with the telescoping sum
$$
1-q^M= 1-q^{m_1B_1}+q^{m_1B_1}(1-q^{m_2B_2})+\cdots +q^{\sum_{i=1}^{k-1}m_iB_i}(1-q^{m_kB_k}).
$$
which implies
\begin{equation}
\label{IEeq}
\begin{aligned}
(1&-q^M)\prod_{i=1}^k (1-q^{B_i})^{-1}\\
=&\sum_{i=1}^k q^{m_1B_1+\cdots m_{i-1}B_{i-1}}
\prod_{j=1}^{i-1} (1-q^{B_j})^{-1} \frac{1-q^{m_iB_i}}{1-q^{B_i}}
\prod_{j=i+1}^{k} (1-q^{B_j})^{-1}.
\end{aligned}
\end{equation}

We see that \eqref{IEeq} is the generating function for partitions $\mu$ 
with parts from $\{B_1,B_2,\cdots, B_k\}$ such that $E_\lambda(\mu)=0.$  
The $i^{th}$ term of the sum represents partitions 
$\mu=(B_1^{n_1}, B_2^{n_2},\cdots, B_k^{n_k})$ 
$$
n_1\ge m_1, \  n_2\ge m_2, \cdots, n_{i-1}\ge m_{i-1}, \ n_i< m_i.
$$ 
These disjoint sets cover all $\mu$ with $E_\lambda(\mu)=0.$

Adding back the multiples of $\lambda$ by multiplying by $(1-wq^M)^{-1}$, 
and also the unused parts from $\mathbb{A}$, gives the result.
\end{proof}

\begin{defn} Let $\mathbb{A}$ be a set of parts. If $\lambda$ has 
parts from $\mathbb{A},$ let 
$E_\lambda^{\mathbb{A}}(n,k)$ be the number of partitions $\mu$ of $n$ 
with parts from $\mathbb{A}$ such that $E_\lambda(\mu)=k.$
\end{defn}

\begin{cor} 
\label{oddcor} For any set of part sizes $\mathbb{A}$, 
let $\lambda_1$ and $\lambda_2$ be two partitions of $M$ into parts 
from $\mathbb{A}.$ Then for all $n,k\ge 0$ 
$$
E_{\lambda_1}^{\mathbb{A}}(n,k)=E_{\lambda_2}^{\mathbb{A}}(n,k).
$$
\end{cor}

Here are the promised versions of Theorem~\ref{RRxcomb1} and Theorem~\ref{RRx2comb1} 
when $M$ does not satisfy the$\mod 5$ condition.

\begin{cor} 
\label{extendRRcomb}
Suppose that $\lambda$ is a partition of $M$ into parts 
congruent to $1$ or $4 \mod 5.$ Then Theorem~\ref{RRxcomb1} holds if the 
number of partitions having $M$ of multiplicity $k$ 
is replaced by $E_{\lambda}^{\mathbb{A}}(n,k)$, 
$\mathbb{A}=\{1,4,6,9, \cdots \}.$ Also, if $\lambda$ is a partition of $M$ into parts 
congruent to $2$ or $3 \mod 5,$ then Theorem~\ref{RRx2comb1} holds if the 
number of partitions having $M$ of multiplicity $k$  is replaced by 
$E_{\lambda}^{\mathbb{B}}(n,k)$, 
$\mathbb{B}=\{2,3,7,8, \cdots \}.$   
\end{cor}

\begin{example} Let $\lambda=(6,1)$, $M=7$, and $n=17.$ The 
equinumerous sets of partitions for Corollary~\ref{extendRRcomb} are
$$
\begin{aligned}
&\{ (9,6,1,1), (6,6,4,1), (6,4,4,1,1,1), (6,4,1^7), (6,1^{11})\}
\leftrightarrow \\
&\{ (16,1), (15,2), (14,3), (13,4), (13,3,1)\}.
\end{aligned}
$$
\end{example}

One corollary of the Rogers-Ramanujan identities is that there are more partitions 
of $n$ into parts congruent to $1$ or $4\mod 5$ than into parts 
congruent to $2$ or $3\mod 5.$ Kadell \cite{K} gave an injection which proves this, 
and Berkovich-Garvan \cite[Theorem 5.1]{BG} gave a stronger injection for modulo 8.
We can use Corollary~\ref{oddcor}, Theorem~\ref{RR1x}, and Theorem~\ref{RR2x} 
to generalize this fact.

\begin{theorem} Let
$$
\mathbb{A}=\{ 5k+1, 5k+4: k\ge 0\}, \qquad
\mathbb{B}=\{ 5k+2, 5k+3: k\ge 0\}. 
$$
Fix partitions $\lambda\vdash M$ and $\theta\vdash M$, $M\ge 3$, 
with parts from $\mathbb{A}$ and $\mathbb{B}$ respectively. 
Then for all $n,k\ge 0$
$$
E_{\theta}^{\mathbb{B}}(n,k)\le E_{\lambda}^{\mathbb{A}}(n,k).
$$ 
\end{theorem}

\begin{proof} By Corollary~\ref{oddcor}, Theorem~\ref{RR1x}, and Theorem~\ref{RR2x} we have
$$
\begin{aligned}
\sum_{n=0}^\infty &\sum_{k=0}^\infty q^n w^k ( E_{\lambda}^{\mathbb{A}}(n,k)-E_{\theta}^{\mathbb{B}}(n,k))\\
=& 
\frac{q-q^{M+1}}{1-wq^M}+\sum_{k=2}^\infty q^{k^2} \frac{[M]_q}{1-wq^M}\frac{1}{(q^2;q)_{k-2}}.
\end{aligned}
$$
All terms are positive except for the first term. If we add the $k=2$ 
term to the first term we have
$$
\frac{q-q^{M+1}+q^4[M]_q}{1-wq^M}
$$
whose numerator is positive for $M\ge 3.$
\end{proof}

\section{Remarks} 

In \cite{O:S} marked versions of the 2nd Rogers-Ramanujan identity are given for
\begin{enumerate}
\item a single part $\{M\},$
\item two parts $\{2,M\},$
\item four parts $\{2,3,7,8\}$.
\end{enumerate}

We do not have a general version of Proposition~\ref{mainprop} 
which gives the last marked version.

A $q$-analogue of Euler's odd=distinct theorem \cite[Theorem 1]{S} is the following.
Let $q$ be a positive integer.
The number of partitions of $N$ into
$q$-odd parts $[2k+1]_q$ is equal to the the number of partitions of $N$ into parts 
$[m]_q$ whose multiplicity is $\le q^m$. A generating function 
identity equivalent to this result is
$$
\prod_{n=0}^\infty \frac{1}{1-t^{[2n+1]_q}}=
1+\sum_{m=1}^\infty t^{[m]_q} \frac{1-t^{q^m[m]_q}}{1-t^{[m]_q}}\prod_{k=1}^{m-1}
\frac{1-t^{(q^k+1)[k]_q}}{1-t^{[k]_q}}.
$$
We do not know how to perturb this identity to mark a part.

Given $\lambda$ and $\mu$, $E_\lambda(\mu)$ is an integer which counts the 
number of $\lambda$'s in $\mu$. One could imagine defining instead 
a rational value for this ``multiplicity".

\end{document}